\documentclass[a4paper,11pt]{article}
\title{Bulging Triangles: Generalization of Reuleaux Triangles}
\author{{\sc Norihiro Someyama}$^*$}
\date{{\small $^*$Shin-yo-ji Buddhist Temple, Tokyo, Japan \\ \vspace{1mm}
{\tt philomatics@outlook.jp}\\ \vspace{1mm}
ORCID iD: https://orcid.org/0000-0001-7579-5352}}
\usepackage{latexsym,amsmath,amssymb,amsthm}
\usepackage{amscd,mathrsfs}
\usepackage{bm}
\usepackage[pdftex]{graphicx}
\usepackage{fancybox}
\usepackage{ascmac}
\usepackage{multicol}
\usepackage{tabularx}
\usepackage[top=30mm,bottom=24mm,left=19mm,right=19mm]{geometry}
\AtBeginDvi{}
\newtheorem{thm}{Theorem}[section]
\newtheorem{prop}{Proposition}[section]

\newtheorem{df}{Definition}[section]

\theoremstyle{definition}

\newtheorem{rem}{Remark}[section]

\newcommand{\arct}{
\stackrel{
\rotatebox{-90}{{\tiny {\bf (}\hspace{-0.5mm}}}
}
{\triangle}\hspace{-1mm}
}
\makeatletter

\@addtoreset{equation}{section}
\makeatother

\begin{document}
\maketitle

\begin{abstract}
We introduce a bulging triangle like the generalization of the Reuleaux triangle.
We may be able to propose various ways to bulge a triangle, but this paper presents the way so that its vertices are the same as them of the original triangle.
We find some properties and theorems of our bulging triangles.
In particular, we investigate, via calculus, whether basic facts such as triangle inequalities and Pythagorean theorem hold for bulging triangles.
\end{abstract}

{\small
{\bf Keywords}: Reuleaux triangle, Covexity, Triangle inequality, Pythagorean theorem.
}
\vspace{2mm}

{\small
{\bf 2010 Mathematics Subject Classification}: 51M04, 51M05, 51N20.
}

\begin{multicols}{2}
\section{Introduction}
There is the Reuleaux triangle (also called the spherical triangle) as one of famous and notable figures in geometry.
It is named after F. Reuleaux (1829-1905), a German mechanical engineer.
He considered a figure in which an equilateral triangle is bulged in a natural way so that it has a constant width.
To be more precise, it is defined as follows:

Let $\triangle{\rm ABC}$ be an equilateral triangle consisting of vertices ${\rm A}$, ${\rm B}$ and ${\rm C}$.
It namely holds that 
\begin{align*}
|\overline{{\rm AB}}|
=|\overline{{\rm BC}}|
=|\overline{{\rm CA}}|
\end{align*}
where $\overline{{\rm AB}}$ denotes the side whose end points are ${\rm A}$ and ${\rm B}$, $|\overline{{\rm AB}}|$ denotes its length, and the other edges are similar.
We have a bulging equilateral triangle like Figure \ref{fig:RT} by drawing arcs with center ${\rm C}$ and radius $\overline{{\rm BC}}$, with center ${\rm A}$ and radius $\overline{{\rm CA}}$ and with center ${\rm B}$ and radius $\overline{{\rm AB}}$.

\begin{figure*}[h]
\begin{center}
\includegraphics[width=4cm]{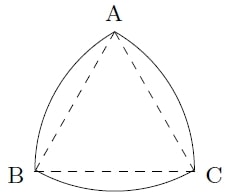}
\end{center}
\caption{Reuleaux Triangle}
\label{fig:RT}
\end{figure*}

The most notable property of the Reuleaux triangle is that it can rotate completely while being within the square and touching all four sides of the square.
This property has been applied to architectural and engineering technologies.

Any Reuleaux triangle has some trivial mathematical properties, e.g., 
\begin{itemize}
\item it is (strictly) convex,
\item the ``triangle inequality'' holds, i.e., the sum of the lengths of any two ``round sides'' is greater than the length of the remaining one ``round side'',
\item the three ``round sides'' are equal and the common length is $\pi a/3\fallingdotseq 1.047a$ if the length of one side of $\triangle{\rm ABC}$ is $a$,
\item its area is $(\pi-\sqrt{3})a^2/2\fallingdotseq 1.628\times (\sqrt{3}a^2/4)$ (i.e., about 1.628 times the original area) if the length of one side of $\triangle{\rm ABC}$ is $a$.
\end{itemize}
Moreover, Reuleaux triangles are widely applied in various fields and there is no end to the list.
See, e.g., \cite{BQB} for the benefits and applications of Reuleaux triangles. 
It is also well known that the Reuleaux triangle is one of solutions for the Kakeya problem (see, e.g., \cite{B} for more information).

Our aim in this paper is to introduce the ``generalized'' Reuleaux triangle with different lengths on the three ``round sides'' and study the properties.

\section{Main Topics}

\subsection{Definition of Bulging Triangles}
Let $\triangle{\rm ABC}$ be a triangle.
For convenience, think of this as an acute triangle.
We would like to make a figure that $\triangle{\rm ABC}$ is bulged by the following way.
Such a figure is called the {\bf bulging triangle (stemmed) from $\triangle{\rm ABC}$} in this paper.

\begin{figure*}[h]
\begin{center}
\includegraphics[width=5cm]{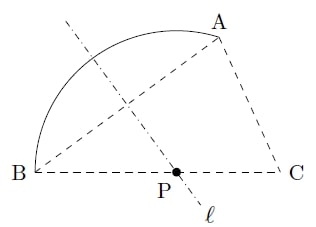}
\end{center}
\caption{${\rm AB}$-center ${\rm P}$}
\label{fig:GRT1}
\end{figure*}

Consider the perpendicular bisector $\ell$ of side $\overline{{\rm AB}}$, and find out the intersection point ${\rm P}$ of another side and $\ell$.
Then, remark that ${\rm P}$ lies on the longer ($\overline{{\rm BC}}$ in the case of Figure \ref{fig:GRT1}) of the remaining two sides $\overline{{\rm BC}}$ and $\overline{{\rm CA}}$.
This fact is almost clear, but we will prove it later (Proposition \ref{prop:ABcenter}).
Determine intersection points ${\rm Q}$ and ${\rm R}$ in the same way for sides $\overline{{\rm BC}}$ and $\overline{{\rm CA}}$.
Here each ${\rm P}$, ${\rm Q}$, ${\rm R}$ is, of course, uniquely determined.
Note that
\begin{align*}
|\overline{{\rm AP}}|&=|\overline{{\rm BP}}|, \\
|\overline{{\rm BQ}}|&=|\overline{{\rm CQ}}|, \\
|\overline{{\rm CR}}|&=|\overline{{\rm AR}}|. 
\end{align*}
We obtain the ``round side'' of $\triangle{{\rm ABC}}$ by drawing the circular arc whose center and radius are ${\rm P}$ and $|\overline{{\rm AP}}|\ (=|\overline{{\rm BP}}|)$, respectively.
The other ``round sides'' are similarly obtained.
From the above, we have been able to define the ``bulging triangle'' (Figure \ref{fig:GRT2}) that is naturally extended the Reuleaux triangle to the general case.
We can treat vertices of the original triangle as vertices of the bulging triangle.

\begin{figure*}[h]
\begin{center}
\includegraphics[width=5cm]{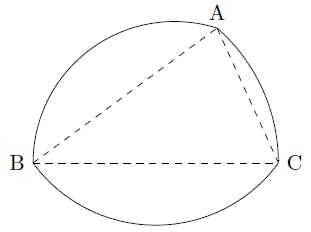}
\end{center}
\caption{Bulging Triangle}
\label{fig:GRT2}
\end{figure*}

In what follows, we set the following promises:
\begin{itemize}
\item We write $\arct{\rm ABC}$ for the bulging triangle whose vertices are ${\rm A}$, ${\rm B}$ and ${\rm C}$.
\item Its ``round sides'' are called {\bf edges} of $\arct{\rm ABC}$ and are denoted by $\widetilde{{\rm AB}}$, $\widetilde{{\rm BC}}$ and $\widetilde{{\rm CA}}$.
Their lengths are denoted by $|\widetilde{{\rm AB}}|$, $|\widetilde{{\rm BC}}|$ and $|\widetilde{{\rm CA}}|$, respectively.
\item Three points ${\rm P}$, ${\rm Q}$ and ${\rm R}$ defined above are called {\bf centers associated with edges of $\arct{\rm ABC}$}.
We often abbreviate such ${\rm P}$ as the {\bf ${\rm AB}$-center} and the others are similar.
\item The radius (i.e., $\overline{{\rm AP}}$, $\overline{{\rm BP}}$, and so on) of the sector consisting of ${\rm A}$, ${\rm B}$ and ${\rm P}$ is called the {\bf ${\rm AB}$-radius}.
The others are similar.
\item The ``length of the radius'' is simply abbreviated as ``the radius'' throughout this paper.
\end{itemize}

\begin{rem}
If $\triangle{\rm ABC}$ is an equilateral triangle, the ${\rm AB}$-center is ${\rm C}$, the ${\rm BC}$-center is ${\rm A}$ and the ${\rm CA}$-center is ${\rm B}$.
Hence, the bulging triangle defined above is certainly an extension of the Reuleaux triangle.
$\blacklozenge$
\end{rem}

\subsection{Basic Properties of Bulging Triangles}

\begin{prop}
\label{prop:ABcenter}
The ${\rm AB}$-center ${\rm P}$ of the bulging triangle $\arct{\rm ABC}$ is on the longer of the remaining two sides $\overline{{\rm BC}},\overline{{\rm CA}}$.
In other words, ${\rm P}$ is the internally dividing point of either $\overline{{\rm BC}}$ or $\overline{{\rm CA}}$.
The ${\rm BC}$-center and ${\rm CA}$-center are similar.
\end{prop}

\begin{proof}
Assume $\angle{\rm ABC}<\angle{\rm BAC}$ and denote by ${\rm P}$ the point such that $|\overline{{\rm AP}}|=|\overline{{\rm BP}}|$.
Then, 
\begin{align*}
\angle{\rm BAP}=\angle{\rm ABC}<\angle{\rm BAC}
\end{align*}
and so ${\rm P}$ is on $\overline{{\rm BC}}$.
Since ${\rm P}$ is the ${\rm AB}$-center of $\arct{\rm ABC}$, this completes the proof.
\end{proof}

The bulging triangle that dents at a vertex is not ideal.
We need to check the following fact.

\begin{prop}
\label{prop:BTconv}
Let $\arct{\rm ABC}$ be any bulging triangle from $\triangle{\rm ABC}$.
\begin{itemize}
\item[1)] If $\triangle{\rm ABC}$ is acute, then $\arct{\rm ABC}$ is convex;
\item[2)] If $\triangle{\rm ABC}$ is right, then $\arct{\rm ABC}$ is convex;
\item[3)] If $\triangle{\rm ABC}$ is obtuse, then $\arct{\rm ABC}$ is concave.
\end{itemize}
\end{prop}

\begin{proof}
Without loss of generality, we can visualize $\arct{\rm ABC}$ from $\triangle{\rm ABC}$ with $\angle{\rm BCA}=\angle R$.
Note that the ${\rm BC}$-center and the ${\rm CA}$-center match and the common point is the middle point of $\overline{{\rm AB}}$.
In other words, $\widetilde{{\rm BC}}+\widetilde{{\rm CA}}$ (i.e., the arc formed by connecting $\widetilde{{\rm BC}}$ and $\widetilde{{\rm CA}}$) coincides with semi-circular arc $\stackrel{\rotatebox{-90}{(}}{{\rm BA}}$ containing point ${\rm C}$.
By virtue of the above, the proof is easily obtained.
\end{proof}

In response to the result of Proposition \ref{prop:BTconv}, {\it we exclusively consider bulging triangles from acute or right-angled triangles in this paper}.

For convenience, we put the following term.

\begin{df}
Let $\triangle{\rm ABC}$ be a right-angled triangle with $\angle{\rm BCA}=\angle R$.
The bulging triangle from this $\triangle{\rm ABC}$ is denoted by $\arct_{R}{\rm BCA}$.
The others are similar.
Such a bulging triangle is called the {\bf right-angled bulging triangle} in this paper.
In particular, we call $\arct_{R}{\rm BCA}$ the {\bf isosceles right-angled bulging triangle} if $|\overline{{\rm BC}}|=|\overline{{\rm CA}}|$.
\end{df}

We study the edges of a bulging triangle from the famous right triangle.
In particular, the result for an isosceles right-angled bulging triangle can be shown to be exactly the same as the conventional one.

\begin{prop}
\label{prop:edgeratio}
For $\arct_{R}{\rm BCA}$ with $|\overline{{\rm BC}}|=a$, it follows that
\begin{itemize}
\item[1)] one has 
\begin{align*}
|\widetilde{{\rm AB}}|:|\widetilde{{\rm BC}}|:|\widetilde{{\rm CA}}|=\sqrt{2}:1:1
\end{align*}
if $\angle{\rm ABC}=\pi/4$,
\item[2)] one has 
\begin{align*}
|\widetilde{{\rm AB}}|:|\widetilde{{\rm BC}}|:|\widetilde{{\rm CA}}|=4:\sqrt{3}:2\sqrt{3}
\end{align*}
if $\angle{\rm ABC}=\pi/3$.
\end{itemize}
\end{prop}

\begin{proof}
We leave the proof of 1) to the reader and prove only 2).

\begin{figure*}[h]
\begin{center}
\includegraphics[width=3.5cm]{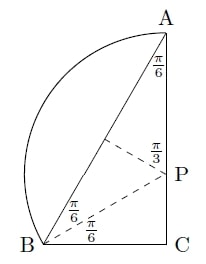}
\end{center}
\caption{Right-Angled Bulging Triangle}
\label{fig:GRT3}
\end{figure*}

Since $|\overline{{\rm BC}}|=a$ and $\angle{\rm ABC}=\pi/3$, we have 
\[
|\overline{{\rm AB}}|=2a,\quad
|\overline{{\rm CA}}|=\sqrt{3}a,\quad
\angle{\rm BAC}=\frac{\pi}{6}.
\]
Denoting by ${\rm P}$ the ${\rm AB}$-center, ${\rm P}$ lies on $\overline{{\rm CA}}$, the ${\rm AB}$-radius is $2a/\sqrt{3}$ and $\angle{\rm APB}=2\pi/3$.
Hence we have 
\begin{align}
\label{eq:wAB4}
|\widetilde{{\rm AB}}|=\frac{4\sqrt{3}}{9}\pi a.
\end{align}
(See Figure \ref{fig:GRT3}.)
Denoting by ${\rm Q}$ the ${\rm BC}$-center, ${\rm Q}$ lies on $\overline{{\rm AB}}$, the ${\rm BC}$-radius is $a$ and $\angle{\rm BPC}=\pi/3$.
Hence we have 
\begin{align}
\label{eq:wBCs3}
|\widetilde{{\rm BC}}|=\frac{1}{3}\pi a.
\end{align}
Denoting by ${\rm R}$ the ${\rm CA}$-center, ${\rm R}={\rm Q}$, the ${\rm CA}$-radius is $a$ and $\angle{\rm CPA}=2\pi/3$.
Hence we have 
\begin{align}
\label{eq:wCA2s3}
|\widetilde{{\rm CA}}|=\frac{2}{3}\pi a.
\end{align}
(\ref{eq:wAB4})-(\ref{eq:wCA2s3}) imply the claim, so the proof is complete.
\end{proof}

We investigate the relationship between the edge of the bulging triangle and the angle of the original triangle.
Take note that we can meet exactly the same results for a usual triangle.

\begin{prop}
Let $\arct{\rm ABC}$ be an isosceles (but not necessarily right-angled) bulging triangle.
That is, suppose 
\[
\angle{\rm ABC}=\angle{\rm BCA},
\quad i.e.,\quad 
|\overline{{\rm AB}}|=|\overline{{\rm CA}}|.
\]
Then, one has 
\[
|\widetilde{{\rm AB}}|=|\widetilde{{\rm CA}}|.
\]
\end{prop}

\begin{figure*}[h]
\begin{center}
\includegraphics[width=3cm]{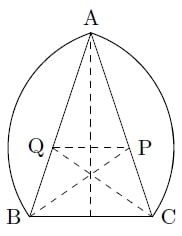}
\end{center}
\caption{Isosceles Right-Angled Bulging Triangle}
\label{fig:GRT4}
\end{figure*}

\begin{proof}
Denote by ${\rm P}$ and ${\rm Q}$ the ${\rm AB}$-center and ${\rm CA}$-center of $\arct{\rm ABC}$, respectively (see Figure \ref{fig:GRT4}).
It is obviously that $\overline{{\rm QP}}\parallel \overline{{\rm BC}}$.
Considering the perpendicular bisector $L$, passing through point ${\rm A}$, of $\overline{{\rm BC}}$, it follows that ${\rm Q}$ is the point of symmetry of ${\rm P}$ with respect to $L$.
${\rm B}$ is also, of course, the point of symmetry of ${\rm C}$ with respect to $L$.
Hence, the claim clearly holds.
This completes the proof.
\end{proof}

\begin{prop}
\label{prop:ang-leng}
Suppose that $\triangle{\rm ABC}$ obeys
\begin{align}
\label{eq:angassump}
0<\angle{\rm ABC}
<\angle{\rm BCA}
\le \frac{\pi}{2}.
\end{align}
Then, $\arct{\rm ABC}$ from such $\triangle{\rm ABC}$ satisfies that 
\begin{align}
\label{eq:wCAlewAB}
|\widetilde{{\rm CA}}|<|\widetilde{{\rm AB}}|.
\end{align}
\end{prop}

\begin{proof}
Put
\begin{align}
\label{eq:abc}
a:=|\overline{{\rm BC}}|,\quad
b:=|\overline{{\rm CA}}|,\quad
c:=|\overline{{\rm AB}}|
\end{align}
and
\begin{align}
\label{eq:abg}
\alpha:=\angle{\rm BAC},\quad
\beta:=\angle{\rm ABC},\quad
\gamma:=\angle{\rm BCA}.
\end{align}
We prove tha claim by dividing into two patterns.
\begin{itemize}
\item[i)] The case of $\alpha\le \beta$:

By noting that $\beta<\gamma$, we have
\begin{align*}
|\widetilde{{\rm CA}}|=\frac{b}{2\cos \alpha}(\pi-2\alpha)
\end{align*}
and
\begin{align*}
|\widetilde{{\rm AB}}|=\frac{c}{2\cos \alpha}(\pi-2\alpha).
\end{align*}
From $\beta<\gamma$ again, we obtain $b<c$.
This implies (\ref{eq:wCAlewAB}).
\item[ii)] The case of $\beta\le \alpha$:

By noting that $\beta<\gamma$, we have
\begin{align*}
|\widetilde{{\rm CA}}|=\frac{b}{2\cos \alpha}(\pi-2\alpha)
\end{align*}
and
\begin{align*}
|\widetilde{{\rm AB}}|=\frac{c}{2\cos \beta}(\pi-2\beta).
\end{align*}
From $\beta<\gamma$ again, we obtain $b<c$.
In order to gain (\ref{eq:wCAlewAB}), we should prove
\begin{align}
\label{eq:p2acap2bcb}
\frac{\pi-2\alpha}{\cos \alpha}
<\frac{\pi-2\beta}{\cos \beta}.
\end{align}
Consider the function with respect to $\alpha\in (0,\pi/2)$,
\begin{align*}
F(\alpha):=(\pi-2\beta)\cos \alpha-(\pi-2\alpha)\cos \beta,
\end{align*}
by fixing $\beta$.
Since $0<\beta\le \alpha<\pi/2$ and
\begin{align*}
F'(\alpha)=-(\pi-2\beta)\sin \alpha+2\cos \beta,
\end{align*}
we have 
\begin{align*}
F''(\alpha)=-(\pi-2\beta)\cos\alpha<0.
\end{align*}
We can thus concludes that $F$ is monotone decreasing.
\[
\lim_{\alpha\uparrow \pi/2}F(\alpha)=0
\]
regardless of the value of $\beta$, and so $F(\alpha)>0$ on $(0,\pi/2)$.
Therefore, we have shown (\ref{eq:p2acap2bcb}) and can obtain (\ref{eq:wCAlewAB}).
\end{itemize}
Hence this completes the proof.
\end{proof}

Regarding the converse of Proposition \ref{prop:ang-leng}, it seems that yet another sufficient condition is necessary.
This point is different from the property of a usual triangle.

\begin{prop}
Let $\arct{\rm ABC}$ be a bulging triangle that $\angle{\rm BAC}$ is the minimum internal angle of $\triangle{\rm ABC}$.
Then, $|\widetilde{{\rm AB}}|>|\widetilde{{\rm CA}}|$ implies $\angle{\rm BCA}>\angle{\rm ABC}$.
\end{prop}

\begin{proof}
We adopt the notations, (\ref{eq:abc}) and (\ref{eq:abg}), again.
We should prove that $\beta<\gamma$ if $\alpha<\min\{\beta,\gamma\}$ and $|\widetilde{{\rm CA}}|<|\widetilde{{\rm AB}}|$.
Since $\alpha<\beta$ and $\alpha<\gamma$ by the assumption, we have
\begin{align*}
|\widetilde{{\rm AB}}|&=\frac{c}{2\cos\alpha}(\pi-2\alpha), \\
|\widetilde{{\rm CA}}|&=\frac{b}{2\cos\alpha}(\pi-2\alpha).
\end{align*}
Another assumption, $|\widetilde{{\rm AB}}|>|\widetilde{{\rm CA}}|$, then implies $c>b$, and so we concludes $\gamma>\beta$.
Hence this completes the proof.
\end{proof}

\subsection{Theorems and the Proofs}

Recall that the triangle inequality always holds for any general triangle (see \cite{SB} for the proofs).
Let us first show that a property like the triangle inequality holds for a bulging triangle.

\begin{thm}
\label{thm:BTtriineq}
For any $\arct{\rm ABC}$, one has
\begin{align}
\label{eq:BTtriineq}
|\widetilde{{\rm AB}}|<|\widetilde{{\rm BC}}|+|\widetilde{{\rm CA}}|.
\end{align}
\end{thm}

\begin{proof}
The policy of this proof is similar to Proposition \ref{prop:ang-leng}.
Put (\ref{eq:abc}) and (\ref{eq:abg}) again.
\begin{itemize}
\item[i)] The case of $\alpha\le \beta$:

By noting that $\beta<\gamma$, we have
\begin{align*}
|\widetilde{{\rm BC}}|=\frac{a}{2\cos \beta}(\pi-2\beta),
\end{align*}
\begin{align*}
|\widetilde{{\rm CA}}|=\frac{b}{2\cos \alpha}(\pi-2\alpha)
\end{align*}
and
\begin{align*}
|\widetilde{{\rm AB}}|=\frac{c}{2\cos \alpha}(\pi-2\alpha).
\end{align*}
The triangle inequality derives $b+c>a$, and so 
\begin{align*}
|\widetilde{{\rm CA}}|+|\widetilde{{\rm AB}}|&=\frac{b+c}{2\cos \alpha}(\pi-2\alpha) \\
&>\frac{a(\pi-2\alpha)}{2\cos \alpha}.
\end{align*}
In order to gain (\ref{eq:BTtriineq}), we should prove
\begin{align}
\label{eq:p2aca>p2bcb}
\frac{\pi-2\alpha}{\cos \alpha}
>\frac{\pi-2\beta}{\cos \beta}.
\end{align}
Consider the function with respect to $\alpha\in (0,\pi/2)$,
\begin{align*}
F(\alpha):=-(\pi-2\beta)\cos \alpha+(\pi-2\alpha)\cos \beta,
\end{align*}
by fixing $\beta$.
Since $0<\alpha\le \beta<\pi/2$ and
\begin{align*}
F'(\alpha)=(\pi-2\beta)\sin \alpha-2\cos \beta,
\end{align*}
we have 
\begin{align*}
F''(\alpha)=(\pi-2\beta)\cos\alpha>0.
\end{align*}
We can thus concludes that $F$ is monotone increasing.
We now put
\begin{align*}
G(\beta):=\lim_{\alpha\downarrow 0}F(\alpha)=\pi\cos\beta+2\beta-\pi.
\end{align*}
Then it is easy to see that $G(\beta)>0$ on $(0,\pi/2)$, so we leave this confirmation to the reader.
It namely follows that $F(\alpha)>0$ on $(0,\pi/2)$.
Therefore, we have shown (\ref{eq:p2aca>p2bcb}) and can obtain (\ref{eq:BTtriineq}).
\item[ii)] The case of $\beta\le \alpha$:

The same discussion is repeated, so we leave it to the reader.
\end{itemize}
Hence this completes the proof.
\end{proof}

Let us next verify that ``Pythagorean Theorem'' holds only for a right-angled bulging triangle whose two edges are equal.

\begin{thm}
\label{thm:tPyth}
For $\arct_{R}{\rm BCA}$, it follows that
\begin{align*}
\begin{cases}
|\widetilde{{\rm AB}}|^2\ge |\widetilde{{\rm BC}}|^2+|\widetilde{{\rm CA}}|^2 & {\rm if}\ \angle{\rm ABC}\in [\pi/4,\theta_0], \vspace{1mm}\\
|\widetilde{{\rm AB}}|^2<|\widetilde{{\rm BC}}|^2+|\widetilde{{\rm CA}}|^2 & {\rm if}\ \angle{\rm ABC}\in (\theta_0,\pi/2),
\end{cases}
\end{align*}
where 
\begin{align*}
\theta_0:=\frac{\pi |\overline{{\rm CA}}|}{2(|\overline{{\rm BC}}|+|\overline{{\rm CA}}|)}.
\end{align*}
The equality holds if $\angle{\rm ABC}=\theta_0$ or $|\overline{{\rm BC}}|=|\overline{{\rm CA}}|$, i.e., $\angle{\rm ABC}=\angle{\rm BAC}=\pi/4$.
\end{thm}

\begin{proof}
Put $a:=|\overline{{\rm BC}}|$, $b:=|\overline{{\rm CA}}|$, $c:=|\overline{{\rm AB}}|$.
Without loss of generality, we can set that $a\le b$.

\begin{figure*}[h]
\begin{center}
\includegraphics[width=4.5cm]{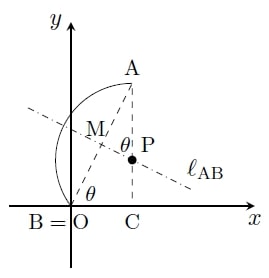}
\end{center}
\caption{Coordinate Representation of Isosceles Right-Angled Bulging Triangle}
\label{fig:GRT5}
\end{figure*}

We prove the claim by algebraic geometry.
Set ${\rm A}(a,b)$, ${\rm B}(0,0)$ and ${\rm C}(a,0)$.
Consider the ${\rm AB}$-perpendicular bisector $\ell_{{\rm AB}}$.
This is given by
\begin{align*}
y=-\frac{a}{b}\left(x-\frac{a}{2}\right)+\frac{b}{2}.
\end{align*}
By finding the intersection point of $\ell_{{\rm AB}}$ and line $x=a$, we have the ${\rm AB}$-center ${\rm P}(a,-a^2/2b+b/2)$.
Thus, the ${\rm AB}$-radius $r_{{\rm AB}}$ is given by
\begin{align*}
r_{{\rm AB}}
=|\overline{{\rm AP}}|
=b-\left(-\frac{a^2}{2b}+\frac{b}{2}\right)
=\frac{a^2+b^2}{2b}.
\end{align*}
Denote by ${\rm M}$ the midpoint of segment $\overline{{\rm AB}}$ and put $\theta:=\angle{\rm APM}=\angle{\rm BPM}$ (see Figure \ref{fig:GRT5}).
It follows $\tan \theta=b/a$, since obviously $\angle{\rm ABC}=\theta$.
Hence,
\begin{align}
\label{eq:tilAB}
|\widetilde{{\rm AB}}|=\frac{a^2+b^2}{b}{\rm Arctan}\,\frac{b}{a}.
\end{align}
Similarly, we obtain
\begin{align}
\label{eq:tilBC}
|\widetilde{{\rm BC}}|=\sqrt{a^2+b^2}\left(\frac{\pi}{2}-{\rm Arctan}\,\frac{b}{a}\right)
\end{align}
and
\begin{align}
\label{eq:tilCA}
|\widetilde{{\rm CA}}|=\sqrt{a^2+b^2}{\rm Arctan}\,\frac{b}{a}.
\end{align}
By virtue of (\ref{eq:tilAB})-(\ref{eq:tilCA}), we have
\begin{align}
\label{eq:tAB2}
|\widetilde{{\rm AB}}|^2
=\left(\frac{a^2+b^2}{b}\right)^2{\rm Arctan}^2\,\frac{b}{a}
\end{align}
and
\begin{align}
\label{eq:tBCtCA}
\begin{aligned}
&|\widetilde{{\rm BC}}|^2+|\widetilde{{\rm CA}}|^2 \\
=&(a^2+b^2)\left(\frac{\pi^2}{4}-\pi {\rm Arctan}\,\frac{b}{a}+2{\rm Arctan}^2\,\frac{b}{a}\right).
\end{aligned}
\end{align}
We here put $t:={\rm Arctan}(b/a)$ and 
\begin{align*}
F(t)&:=\frac{(|\widetilde{{\rm BC}}|^2+|\widetilde{{\rm CA}}|^2)-|\widetilde{{\rm AB}}|^2}{a^2+b^2} \\
&\ =\frac{b^2-a^2}{b^2}t^2-\pi t+\frac{\pi^2}{4}.
\end{align*}
Note that its domain is the interval $[\pi/4,\pi/2)$ from the assumption ``$a\le b$''.
The axis of the parabola $F$ which is convex downward is 
\begin{align*}
t=\alpha:=\frac{\pi b^2}{2(b^2-a^2)}
\end{align*}
and 
\begin{align*}
\frac{b^2}{b^2-a^2}-1=\frac{a^2}{b^2-a^2}>0,
\quad {\rm i.e.,}\quad 
\frac{\pi}{2}<\alpha.
\end{align*}
Moreover, we have
\begin{align*}
F(\pi/4)=\frac{\pi^2(b^2-a^2)}{16b^2}\ge 0
\end{align*}
and
\begin{align*}
\lim_{t\uparrow \pi/2}F(t)=-\frac{\pi^2a^2}{4b^2}<0.
\end{align*}
We find, by the intermediate value theorem, that $F$ has a unique zero point
\begin{align*}
\theta_0=\frac{\pi b}{2(a+b)},
\end{align*}
and we concludes that $F\ge 0$ on $[\pi/4,\theta_0]$ and $F<0$ on $(\theta_0,\pi/2)$.
Therefore, the claim has been proved and we have found that ``Pythagorean theorem'' does not generally hold.

We find the condition for the equal sign to hold.
It is obvious that the equality holds if $\theta=\theta_0$ by virtue of the above discussion.
If $a=b$, we have $\theta=\pi/4$ and ${\rm Arctan}(b/a)=1$.
Hence, from (\ref{eq:tAB2}) and (\ref{eq:tBCtCA}), ``Pythagorean theorem'' holds:
\begin{align*}
|\widetilde{{\rm AB}}|^2=|\widetilde{{\rm BC}}|^2+|\widetilde{{\rm CA}}|^2\left(=\frac{\pi^2a^2}{4}\right).
\end{align*}
This completes the proof.
\end{proof}

\begin{rem}
\label{rem:tPtT}
We check that the ``triangle inequality'' holds for right-angled bulging triangle via the proof of Theorem \ref{thm:tPyth}.
Since
\begin{align*}
|\widetilde{{\rm AB}}|
=\frac{a^2+b^2}{b}{\rm Arctan}\,\frac{b}{a}
\end{align*}
and
\begin{align*}
|\widetilde{{\rm BC}}|+|\widetilde{{\rm CA}}|=\frac{\pi}{2}\sqrt{a^2+b^2},
\end{align*}
we need to prove 
\begin{align}
\label{eq:<pi/2}
\frac{\sqrt{a^2+b^2}}{b}{\rm Arctan}\,\frac{b}{a}<\frac{\pi}{2}
\end{align}
so as to see
\begin{align*}
|\widetilde{{\rm AB}}|<|\widetilde{{\rm BC}}|+|\widetilde{{\rm CA}}|.
\end{align*}
We obtain $\pi/4\le {\rm Arctan}(b/a)<\pi/2$ because of the assumption ``$a\le b$''.
The left-hand side of (\ref{eq:<pi/2}) can be rewritten as
\begin{align*}
\frac{\sqrt{a^2+b^2}}{b}{\rm Arctan}\,\frac{b}{a}
=\frac{{\rm Arctan}(b/a)}{\sin({\rm Arctan}(b/a))},
\end{align*}
so it is sufficient to verify that
\begin{align*}
f(x)=\frac{x}{\sin x},\quad x\in [\pi/4,\pi/2),
\end{align*}
is strictly monotone increasing and $f(x)<\pi/2$.
It is easy to see that fact, so we leave it to the reader.
We have (\ref{eq:<pi/2}) as a result.
$\blacklozenge$
\end{rem}

We finally investigate the relationship between the bulging triangle and the circle.
It is well known that any triangle is inscribed in a circle.
We here try a similar argument for right-angled bulging triangles.
The method of the proof of Theorem \ref{thm:tPyth} is useful for this.

\begin{thm}
\label{thm:RBTcirc}
Any $\arct_{R}{\rm BCA}$ is inscribed in a circle whose center is the midpoint ${\rm M}$ of ${\rm A}$ and ${\rm B}$ and whose radius is $|\overline{{\rm AM}}|$.
\end{thm}

\begin{figure*}[h]
\begin{center}
\includegraphics[width=4.5cm]{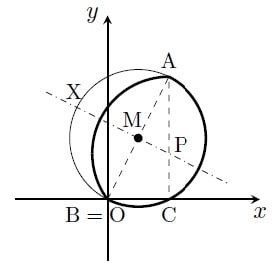}
\end{center}
\caption{Circumscribed Circle of Right-Angled Bulging Triangle}
\label{fig:GRT6}
\end{figure*}

\begin{proof}
We saw that $\widetilde{{\rm BC}}+\widetilde{{\rm CA}}$ is semi-circular arc $\stackrel{\rotatebox{-90}{(}}{{\rm BA}}$ containing ${\rm C}$ in the proof of Proposition \ref{prop:BTconv}.
Moreover, the center and radius of that semi-circular arc are ${\rm M}$ and $|\overline{{\rm AM}}|$, respectively.
Denote by $C$ the circle whose center is ${\rm M}$ and whose radius is $|\overline{{\rm AM}}|$.
Consider the line ${\rm PM}$ (i.e., perpendicular bisector $\ell$ of $\overline{{\rm AB}}$) and the intersection point ${\rm X}$ of $\ell$ and $C$ (see Figure \ref{fig:GRT6}).
Recall
\begin{align}
\label{eq:MPcor}
{\rm M}(a/2,b/2),\quad
{\rm P}(a,(-a^2+b^2)/2b).
\end{align}
Then, it is sufficient to prove that the ${\rm CA}$-radius is shorter than the radius of the circumscribed circle, i.e.,
\begin{align}
\label{eq:APXP}
|\overline{{\rm AP}}|<|\overline{{\rm XP}}|.
\end{align}

It has already been obtained that
\begin{align}
\label{eq:APleng}
|\overline{{\rm AP}}|=\frac{a^2+b^2}{2b}
\end{align}
in the proof of Theorem \ref{thm:tPyth}.
Since
\[
|\overline{{\rm XM}}|
=|\overline{{\rm AM}}|
=\frac{c}{2}
=\frac{\sqrt{a^2+b^2}}{2}
\]
and
\[
|\overline{{\rm MP}}|
=\frac{a\sqrt{a^2+b^2}}{2b}
\]
from (\ref{eq:MPcor}), we have
\begin{align}
\label{eq:XPleng}
|\overline{{\rm XP}}|=|\overline{{\rm XM}}|+|\overline{{\rm MP}}|
=\frac{(a+b)\sqrt{a^2+b^2}}{2b}.
\end{align}
We should thus check
\begin{align}
\label{eq:XP>APaw}
\frac{(a+b)\sqrt{a^2+b^2}}{2b}> \frac{a^2+b^2}{2b}
\end{align}
to prove (\ref{eq:APXP}), according to (\ref{eq:APleng}) and (\ref{eq:XPleng}).
We gain that
\begin{align*}
&\frac{(a+b)\sqrt{a^2+b^2}}{2b}-\frac{a^2+b^2}{2b} \\
&=\frac{\sqrt{a^2+b^2}}{2b}\{(a+b)-\sqrt{a^2+b^2}\},
\end{align*}
but the triangle inequality for $\triangle{{\rm ABC}}$ implies that
\[
(a+b)-\sqrt{a^2+b^2}=(a+b)-c>0.
\]
So (\ref{eq:XP>APaw}), i.e., (\ref{eq:APXP}) has obtained.
This completes the proof.
\end{proof}

\begin{rem}
The best feature of the above proof is to show (\ref{eq:XP>APaw}), i.e.,
\[
a+b>\sqrt{a^2+b^2}.
\]
This was a paraphrase of the triangle inequality.
The triangle inequality seems to have already been considered trivial in ancient Greek times, but it is quite interesting that such a trivial property is the essence of Theorem \ref{thm:RBTcirc}.
$\blacklozenge$
\end{rem}

\section{Conclusion}
We have defined the generalized Reuleaux triangles by one of many ways.
Beside this way, there is for instance a way to construct a generalization of the Reuleaux triangle $\triangle{\rm ABC}$ as the closed and convex curve consisting of three arcs centered on each ${\rm A}$, ${\rm B}$, ${\rm C}$.
That can be seen a little in \cite{R}.
As we can see from this, ``generalization'' has many implications.

We have proved that ``triangle inequalities'' always hold but ``Pythagorean theorems'' do not generally hold for bulging triangles defined in this paper.
Also, right-angled bulging triangles have good qualities (the latter claim of Theorem \ref{thm:tPyth} and Remark \ref{rem:tPtT}).
In that sense, our definition of the generalization of Reuleaux triangles would be moderately appropriate.

We may however be able to define new bulging triangles satisfying the above two properties.
In addition, there are a number of topics left by this paper:
\begin{itemize}
\item We can easily obtain the area of a Reuleaux triangle (see Introduction of this paper), but it seems difficult to find that of a general bulging triangle.
\item It is unclear if any (not necessarily right-angled) bulging triangle is inscribed in a circle.
\item The application of bulging triangles is our concern. etc.
\end{itemize}
We would like to make research on them a future subject.


\end{multicols}

\end{document}